\documentclass[11pt]{amsart}
\usepackage{amssymb,amsfonts,amsmath}
\usepackage[all]{xy}
\usepackage[shortlabels]{enumitem}
\usepackage{amssymb}
\usepackage{lipsum}

\usepackage{hyperref,color}

\theoremstyle{theorem}
\newtheorem{theorem}{Theorem}[section]
\newtheorem{proposition}[theorem]{Proposition}

\newtheorem{corollary}[theorem]{Corollary}

\theoremstyle{definition}

\newtheorem{lemma}[theorem]{Lemma}
\newtheorem{example}[theorem]{Example}

\theoremstyle{remark}
\newtheorem{remark}{Remark}

\theoremstyle{claim}

\theoremstyle{theoremA}
\newtheorem*{theoremA}{Theorem A}
\theoremstyle{theoremB}
  
\newcommand{\Div}{\mbox{div}_M}

\newcommand{\real}{\mathbb{R}}

\newcommand{\hy}{\mathbf{H}}
\newcommand{\Sp}{\mathbf{S}}
\newcommand{\Si}{\Sigma}
\newcommand{\si}{\sigma}
\newcommand{\te}{\theta}

\newcommand{\al}{\alpha}

\newcommand{\om}{\omega}

\newcommand{\na}{\nabla}
\newcommand{\ep}{\epsilon}
\newcommand{\ga}{\gamma}
\newcommand{\Ga}{\Gamma}
\newcommand{\be}{\beta}
\newcommand{\De}{\Delta}
\newcommand{\de}{\delta}

\newcommand{\vp}{\varphi}
\newcommand{\vol}{\text{\rm vol}}

\newcommand{\lan}{\langle}
\newcommand{\ran}{\rangle}
\newcommand{\p}{\partial}

\newcommand{\dv}{\mathrm{div}}
\newcommand{\hs}{\mathrm{Hess}}
\newcommand{\tr}{\mathrm{Tr\,}}

\newcommand{\rad}{\mathrm{rad}}
\newcommand{\m}{\mathcal}
\newcommand{\II}{{I\!I}}
\newcommand{\spp}{\mathrm{supp}\,}

\newcommand{\La}{\Lambda}

\title[The Caffarelli-Kohn-Nirenberg inequality on  submanifolds]{The Caffarelli-Kohn-Nirenberg inequality for submanifolds in Riemannian manifolds}
\author{M. Batista}
\address{Instituto de Matem\'atica, Universidade Fe\-deral de Alagoas, Macei\'o, AL, CEP 57072-970, Brazil}\email{mhbs@mat.ufal.br}
\author{H. Mirandola}
\address{Instituto de Matem\'atica, Universidade Fe\-deral do Rio de Janeiro, Rio de Janeiro, RJ, CEP 21945-970, Brasil} \email{mirandola@ufrj.br}
\author{F. Vit\'orio}
\address{Instituto de Matem\'atica, Universidade Fe\-deral de Alagoas, Macei\'o, AL, CEP 57072-970, Brazil}\email{feliciano.vitorio@gmail.com}
\thanks{The authors are supported by CNPq.}

\begin{document}
\maketitle
\begin{abstract} After works by Michael and Simon \cite{MS}, Hoffman and Spruck \cite{HS}, and White \cite{W}, the celebrated Sobolev inequality could be extended to submanifolds in a huge class of Riemannian manifolds. The universal constant obtained depends only on the dimension of the submanifold. A sort of applications to the submanifold theory and geometric analysis have been obtained from that inequality. It is worthwhile to point out that, by a Nash Theorem,  every Riemannian manifold can be seen as a submanifold in some Euclidean space. In the same spirit, Carron obtained a Hardy inequality for submanifolds in Euclidean spaces.
In this paper, we will prove the Hardy, weighted Sobolev and Caffarelli-Kohn-Nirenberg inequalities, as well as some of their derivatives, as Galiardo-Nirenberg and Heisenberg-Pauli-Weyl inequalities, for submanifolds in a class of manifolds, that include,  the Cartan-Hadamard ones.
\end{abstract}

\tableofcontents

\section{Introduction}

Over the years, geometers have been interested in understanding how integral inequalities imply geometric or topological obstructions on Riemannian manifolds. Under  this purpose,  some integral inequalities lead us to study positive solutions to critical singular quasilinear elliptic problems, sharp constants, existence, non-existence and symmetry results for extremal functions on subsets in the Euclidean space. About these subjects, one can read, for instance, \cite{BT}, \cite{C}, \cite{CC}, \cite{CKN}, \cite{CW}, \cite{HS}, \cite{MS}, \cite{KO} and references therein.

In the literature, some of the most known integral inequalities are the Hardy inequality, Gagliardo-Nirenberg inequality, and, more generally, the Caffarelli-Kohn-Nirenberg inequality. %The validity of these inequalities and their corresponding sharp constants on a given Riemannian manifold measures, in a suitable sense, how close that manifold is to an Euclidean space. 
These inequalities imply comparison for the  volume growth, estimates of the essencial spectrum for the Schr\"{o}dinger operators, parabolicity,  among others properties (see, for instance, \cite{L, dCX, X}).

In this paper, we propose to study the Caffarelli-Kohn-Nirenberg (CKN) inequality for submanifolds in a class of Riemannian manifolds that includes, for instance, the Cartan-Hadamard manifolds, using an elementary and very efficient approach. We recall that a Cartan-Hadamard manifold is a complete simply-connected Riemannian manifold with nonpositive sectional curvature.  Euclidean and hyperbolic spaces are 
 the simplest examples of Cartan-Hadamard manifolds.
\section{Preliminaries}

In this section, let us start recalling some concepts, notations and basic properties about submanifolds. First, let $M=M^k$ be a $k$-dimensional Riemannian manifold with (possibly nonempty) smooth boundary $\p M$. Assume $M$ is isometrically immersed in a complete Riemannian manifold $\bar M$. Henceforth, we will denote by $f:M\to \bar M$ the isometric immersion. In this paper, no restriction on the codimension of $f$ is required. By abuse of notation, sometimes we will identify $f(x)=x$, for all $x\in M$. 
Let $\lan\cdot,\cdot\ran$ denote the Euclidean metric on $\bar M$ and consider the same notation to the metric induced on $M$. Associated to these metrics, consider the Levi-Civita connections $D$ and $\na$ on $\bar M$ and $M$, respectively. It easy to see that $\nabla_Y Z = (D_YZ)^\top,$ where $\top$ means the orthogonal projection onto the tangent  bundle $TM$. The Gauss equation says $$D_Y Z = \nabla_YZ + \II(Y,Z),$$
where $\II$ is a quadratic form named by  {\it second fundamental form}. The mean curvature vector is defined by $H=\tr_M \,\II$.

Let $\m K:[0,\infty)\to [0,\infty)$ be a nonnegative continuous function and  $h\in C^2([0,+\infty))$ the solution of the Cauchy problem:
\begin{equation}\label{cauchy-h}
\begin{array}{l}
h''+ \m K h = 0, \\
h(0)=0, \ h'(0)=1.
\end{array}
\end{equation} 
Let $0<\bar r_0=\bar r_0(\m K)\le +\infty$ be the supremum value where the restriction $h|_{[0,\bar r_0)}$ is increasing and let $[0,\bar s_0)=h([0,\bar r_0))$. Notice that $h'$ is non-increasing since $h''=-\m K h\le 0$.

\begin{example}\label{mK-constant} If $\m K=b^2$, with $b\ge 0$, then \begin{enumerate}[(i)] \item\label{ambient-b=0} if $b=0$, it holds $h(t)=t$ and $\bar r_0=\bar s_0=+\infty$; 
\item\label{ambient-b>0} if $b>0$,  it holds $h(t)=\frac{1}{b}\sin(bt)$ and $\bar r_0=\frac{\pi}{2b}$ and $\bar s_0=h(\bar r_0)=\frac{1}{b}$.
\end{enumerate} \end{example}

For $\xi\in \bar M$,  let $r_{\xi}=d_{\bar M}(\cdot\,,\xi)$ be the distance function on $\bar M$ from $\xi\in \bar M$. In this paper, we will deal with complete ambient spaces $\bar M$  whose radial sectional curvature satisfies  $(\bar K_\rad)_{\xi_0}\le \m K(r_{\xi_0})$, for some fixed $\xi_0\in \bar M$. Let us recall the definition of radial sectional curvature. Let $x\in \bar M$ and, since $\bar M$ is complete, let $\ga:[0,t_0=r_{\xi}(x)]\to \bar M$ be a minimizing geodesic in $\bar M$ from $\xi$ to $x$. For all  orthonormal pair of vectors $Y,Z\in T_x\bar M$ we define $(\bar K_\rad)_{\xi}(Y,Z)=\lan \bar R(Y,\ga'(t_0))\ga'(t_0), Z\ran$. 

\begin{example}\label{example-warped} Let $(P,d\si^2_P)$ be a complete manifold. Consider the manifold $\bar M=[0,r_0)\times P/\sim$, where $(0,y_1) \sim (0,y_2)$, for all $y_1$, and $y_2\in P$, with the following metric:
\begin{equation}\label{warped-metric} \lan \cdot\,,\cdot\ran_{\bar M} = dr^2 + h(r)^2 d\si^2_P.
\end{equation}
Since $h>0$ in $(0,r_0)$, $h(0)=0$ and $h'(0)=1$, it follows  that $\bar M$ defines a Riemannian manifold. If $P=\Sp^{n-1}$ is the round metric, $\lan\,,\ran_{\bar M}$ is called a rotationally invariant metric.  

We fix the point $\xi_0=(0,y)\in \bar M$. The distance $d_{\bar M}((r,y), \xi_0)=r$, for all $(r,y)\in \bar M$.
The curvatura tensor $\bar R$ of $\bar M$ satisfies
\begin{equation} 
\bar R(Y,\p_r)\p_r = \left\{
\begin{array}{l}
-\frac{h''(r)}{h} Y, \mbox{ if } Y \mbox{ is tangent to } P;\\
0, \mbox{ if } Y=\p_r.
\end{array}
\right.
\end{equation}
Hence, the radial sectional curvature $(\bar K_{\rad})_{\xi_0}(\cdot\,,\cdot)=\lan \bar R(\cdot\,,\p_r)\p_r,\cdot\ran$, with basis point $\xi_0$, satisfies $(\bar K_{\rad})_{\xi_0}=\m K(r)$. 

A huge class of metrics are rotationally symmetric:  (i) The Euclidean metric: $\lan\,,\ran_{\real^n}=dr^2+r^2 d\si^2_{\Sp^{n-1}}$, in $[0,\infty)\times \Sp^{n-1}$. (ii)  The spherical metric $\lan\,,\ran_{\Sp^n}=dr^2 + \sin^2(r) d\si^2_{\Sp^{n-1}}$, in $[0,\pi]\times \Sp^{n-1}$. (iii) The Hyperbolic metric: $\lan\,,\ran_{\hy^n}=dt^2+\sinh^{2}(r) d\si^2_{\Sp^{n-1}}$, in $[0,\infty)\times \Sp^{n-1}$; (iv) Some classical examples in general relativity: Schwarzchild metric, De Sitter-Schwarzchild metric, Kottler-Schwarzchild metric, among others.
\end{example}

Assume the radial sectional curvatures of $\bar M$ satisfies $(\bar K_{\rad})_{\xi_0}\le \m K(r)$, where $r=r_{\xi_0}=d_{\bar M}(\cdot\,,\xi_0)$. We fix $0<r_0 < \min\{\bar r_0(\m K),Inj_{\bar M}(\xi_0)\}$ and consider the geodesic ball $\m B=\m B_{r_0}(\xi_0)=\{x\in \bar M \mid d_{\bar M}(x,\xi_0)<r_0\}$.  It follows that $r$ is differentiable at all points in $\m B^*=\m B\setminus \{\xi_0\}$ and, 
%$$$$ The distance function $r=r_{\xi_0}=d_{\bar M}(\cdot,\xi_0)$ is Lipschitz in $\bar M$ and differentiable in $\bar M\setminus [C(\xi_0)\cup \{\xi_0\}]$, where $C(\xi_0)$ denotes the cut locus of $\xi_0$ in $\bar M$.  Furthermore, the gradient vector $\bar\na r$ on $\bar M$ also satisfies $|\bar\na r|=1$, where it exists. Let $\m B_{r_0}(\xi_0)$ be the geodesic ball in $\bar M$ of radius $r_0=r_0(\m K)$ (as defined in (\ref{cauchy-h})) and centered at $\xi_0$, and let $\m B^*_{r_0}(\xi_0)=\m B_{r_0}(\xi_0)\setminus \{\xi_0\}$. Assuming $(\bar K_{\rad})_{\xi_0}\le \m K(r)$, by the Rauch comparison theorem and Example  \ref{example-warped} above, it follows  that $\m B^*_{r_0}(\xi_0)$ does not intersect the conjugate locus of $\bar M$ but might intersect $C(\xi_0)$. If the last case happens, there will exist $\xi\in \m B^*_{r_0}(\xi_0)$ and two distinct minimizing geodesics of $\bar M$ joining $\xi_0$ to $\xi$.  $$$$ $$$$ In particular, if $\m B$ is simply-connected then $r\in C^2(\m B)$. In this paper, we will always assume that $r$ is $C^2$ at the point of $M$. 
%Furthermore, assuming further $(\bar K_{\rad})_{\xi_0}\le \m K(r)$, 
by the Hessian comparison theorem (see Theorem 2.3 page 29 of \cite{PRS}), we have 
\begin{equation}\label{hes-comp}
\hs_{\,r}(v,v) \ge \frac{h'(r)}{h(r)}(1-\lan \bar\na r,v\ran^2),
\end{equation} 
for all points in $\m B^*$ and vector fields $v:\m B^*\to T\bar M$ with $|v|=1$. 

For a vector field $Y:M\to T\bar M$, the  divergence of $Y$ on $M$ is given by
$$\Div Y = \sum_{i=1}^k \langle D_{e_i}Y, e_i\rangle,$$
where $\{e_1,\cdots, e_k\}$ denotes a local orthonormal frame on $M$.  By simple computations, one has 

\begin{lemma}\label{prop} Let $Y:M\to T\bar M$ be a vector field and $\psi \in C^1(M)$. The following items hold
\begin{enumerate}[(a)]
\item\label{prop-a} $\Div Y = \Div Y^\top - \langle \vec{H}, Y\rangle;$
\medskip
\item\label{prop-b} $\Div (\psi Y) = \psi~\Div Y + \lan \na^M\psi, Y\rangle$.
\end{enumerate}
\end{lemma}

From now on, we will  consider the radial vector field $X=X_{\xi_0}=h(r)\bar\na r$, defined in $\m B^*$. Notice that $|X|=h(r)>0$ everywhere in $\m B^*$.

\begin{lemma}\label{formula}
For all $\al\in (-\infty, +\infty)$, it holds 
\begin{equation*}
\dv_M(\frac{X}{|X|^\al}) \ge h'(r)[\frac{k-\al}{{h(r)}^\al} + \al~\frac{|\bar\na r ^\perp|^2}{{h(r)}^\al}],
\end{equation*}
in $M\cap \m B^*$. Here,  $(\cdot)^\perp$ denotes the orthogonal projection on the normal bundle $TM^\perp$ of $M.$
\end{lemma}
\begin{proof}
By Lemma  \ref{prop} item \ref{prop-b}, $\dv_M (\frac{X}{{h(r)}^\al}) = \frac{1}{{h(r)}^\al}\dv X +\lan \na^M (\frac{1}{{h(r)}^\al}),~X \ran$. Since $1=|\bar\na r|^2=|\bar\na r^\top|^2+|\bar\na r^\perp|^2$, and $\na^M (\frac{1}{|X|^\al}) = -\al~\frac{h'(r) \bar\na r ^\top}{{h(r)}^{\al + 1}}$, one has 
\begin{eqnarray*}
\Div (\dfrac{X}{|X|^\al}) &=& \frac{1}{{h(r)}^\al}\dv_M X -\frac{\al h'(r)}{{h(r)}^{\al+1}}\lan \bar\na r^\top, h(r)\bar\na r\ran 
\\&=& \frac{1}{{h(r)}^\al}\dv_M X -\frac{\al h'(r)}{{h(r)}^\al} + \frac{\al h'(r)|\bar\na r^\perp|^2}{{h(r)}^{\al+1}}.
\end{eqnarray*}
On the other hand, let $\{e_1,\cdots, e_k\}$ denote an orthonormal frame on $M$. By (\ref{hes-comp}), we have
\begin{eqnarray*}
\dv_M X &=& \sum_{i=1}^k\lan D_{e_i}X, e_i\ran =\sum_{i=1}^k [h'(r)\lan \bar\na r,e_i\ran^2 + h(r)\hs_{r}(e_i,e_i)] 
\\&\ge& h'(r)|(\bar\na r)^\top|^2 + h'(r)(k- |(\bar\na r)^\top|^2) = k h'(r)
\end{eqnarray*}
Lemma \ref{formula} follows.
\end{proof}

\section{The Hardy inequality for submanifolds}

Carron \cite{C} proved the following Hardy Inequality.
\begin{theoremA}[Carron]\label{carron-teo} Let $\Si^k$ be a complete non compact Riemannian manifold isometrically immersed in a  Euclidean space $\real^n$. Fix   $v\in \real^n$ and let $r(x)=|x-v|$, for all $x\in \Si$. Then, for all smooth function $\psi\in C_c^\infty(\Si)$ compactly supported in $\Si$, the following Hardy inequality holds:
\begin{equation*} \frac{(k-2)^2}{4} \int_\Si \frac{\psi^2}{r^2} \le \int_\Si [|\na^\Si \psi|^2 + \frac{k-2}{2}\frac{|H|\psi^2}{r}].
\end{equation*}
\end{theoremA}

%Let $C^1(M)$ denote the set of all $C^1$ functions on $M$. %We recall that we are assuming that $M$ is a compact manifold with (possibly nonempty) smooth boundary $\p M$ and that $M$ is isometrically immersed in a complete Riemannian manifold $\bar M$. 
Just comparing Theorem A with Corollary \ref{hardy-hadamard-cor} below,  given $\psi\in C_c^\infty(\Si)$, let $M$ be a compact subset of $\Si$ with compact smooth boundary $\p M$ satisfying $\spp(\psi)\subset M\subset \Si$. We will see that Corollary \ref{hardy-hadamard-cor} does not generalize Theorem A, unless $\Si$ is a minimal submanifold. \\

The result  below will be fundamental to obtain our Hardy inequality (see Theorem \ref{teo-hardy}).

\begin{proposition} \label{teo-hardy-norm} Fixed $\xi_0\in \bar M$, we assume $(\bar K_{\rad})_{\xi_0}\le \m K(r)$, where $r=r_{\xi_0}=d_{\bar M}(\cdot\,,\xi_0)$. Assume further  $M$ is contained in a ball $\m B=\m B_{\bar r_0}(\xi_0)$, for some $0<r_0<\min\{\bar r_0(\m K),\, Inj_{\bar M}(\xi_0)\}$.  Let $1< p < \infty$ and $-\infty<\ga<k$. Then, for all $\psi\in C^1(M)$, with $\psi\ge 0$, it holds
\begin{eqnarray*}
\frac{(k-\ga)^p h'(r_0)^{p-1}}{p^p} \int_{M} \frac{\psi^p h'(r)}{h(r)^{\ga}} + \frac{\ga[(k-\ga)h'(r_0)]^{p-1}}{p^{p-1}}\int_M \frac{\psi^ph'(r)\,|(\bar\na r)^\perp|^2}{h(r)^\ga} 
\\&& \hspace{-13.0cm}\le
\int_M \frac{1}{h(r)^{\ga-p}}[|\na^M \psi|^2 + \frac{\psi^2|H|^2}{p^2}]^{{p}/{2}} + \frac{[(k-\ga)h'(r_0)]^{p-1}}{p^{p-1}} \int_{\p M} \frac{\psi^p}{h(r)^{\ga-1}}\lan \bar\na r,\nu\ran,
\end{eqnarray*}
provided that $\int_{\p M} \frac{\psi^p}{h(r)^{\ga-1}}\lan \bar\na r,\nu\ran$ exists. Here, $\nu$ denotes the outward conormal vector to  $\p M$.
\end{proposition}
\begin{proof} First, we assume $\xi_0\notin M$. Let $X=h(r_{\xi_0})\bar\na r_{\xi_0}$ and write $\ga=\al+\be+1$ with $\al,\be\in \real$.  Let  $\psi \in C^1(M)$. By Lemma \ref{formula},
\begin{eqnarray}\label{dv-hardy}
\dv_M(\frac{\psi^p X^\top}{|X|^\ga}) &=& \psi^p \dv_M(\frac{X^\top}{|X|^\ga}) + \lan \na^M \psi^p,\frac{X}{|X|^\ga}\ran
\nonumber\\&=& \psi^p \dv_M(\frac{X}{|X|^\ga}) + \psi^p \lan \frac{X}{|X|^\ga}, H\ran +p\,\psi^{p-1} \lan \na^M \psi, \frac{X}{|X|^\ga}\ran 
\nonumber\\&\ge& \psi^p h'(r)[\frac{k-\ga}{{h(r)}^\ga}+\frac{\ga |\bar\na r^\perp|^2}{{h(r)}^\ga}]  + \lan p \na^M \psi + \psi {H}, \frac{\psi^{p-1} X}{|X|^\ga}\ran.
\end{eqnarray}

By the divergence theorem, 
\begin{eqnarray}\label{0notin}
 \int_{M} \psi^p h'(r)[\frac{k-\ga}{{h(r)}^{\ga}} + \frac{\ga|\bar\na r^\perp|^2}{{h(r)}^{\ga}}] &\le& - \int_{M} \lan p \na^M \psi + \psi {H}, \frac{\psi^{p-1} X}{|X|^\ga}\ran
\nonumber\\&& +  \int_{\p M} \frac{\psi^p}{|X|^\ga}\lan X,\nu \ran,
\end{eqnarray}
where $\nu$ denotes the outward conormal vector to the  boundary $\p M$ in $M$. 
Let $r^*:=\max_{x\in \spp(\psi)} d_{\bar M}(f(x),\xi_0)$. Since $0<r^* < r_0$ and  $h''=-\m K h\le 0$ it holds that $h'(r)\ge h'(r^*) > h'(r_0)\ge 0$ in $M$.  By the Young inequality with $\ep>0$ (to be chosen soon), it holds
\begin{eqnarray}
\int_{M} \psi^p h'(r)[\frac{k-\ga}{{h(r)}^{\ga}} + \frac{\ga |\bar\na r^\perp|^2}{{h(r)}^\ga}] &\le& -\int_{M} \lan \frac{p\,\na^M\psi}{|X|^\al}+ \frac{\psi{H}}{|X|^\al}, \frac{\psi^{p-1} X}{|X|^{\be+1}}\ran 
\nonumber\\&& + \int_{\p M} \frac{\psi^p}{|X|^\ga}\lan X,\nu \ran
\nonumber\\ &&\hspace{-6.5cm} \le \frac{1}{p\ep^p} \int_M  |\frac{p\na^M\psi}{{h(r)}^\al}+\frac{\psi{H}}{{h(r)}^\al}|^p + \frac{\ep^q}{q} \int_M  \frac{|\psi|^{(p-1)q}}{{h(r)}^{\be q}} + \int_{\p M} \frac{\psi^p}{{h(r)}^{\ga-1}}\lan \bar\na r,\nu \ran
\nonumber\\&&\hspace{-6.5cm} \le \frac{1}{p\ep^p} \int_M |\frac{p^2|\na^M\psi|^2}{{h(r)}^{2\al}}+\frac{\psi^2|{H}|^2}{{h(r)}^{2\al}}|^{p/2} + \frac{\ep^q}{q}  \int_M \frac{|\psi|^{(p-1)q}}{{h(r)}^{\be q}}\frac{h'(r)}{h'(r_0)} 
\nonumber\\&&+  \int_{\p M} \frac{\psi^p}{{h(r)}^{\ga-1}}\lan \bar\na r,\nu \ran, \nonumber
\end{eqnarray}
where $q=\frac{p}{p-1}$. 

Now, consider $\be=(p-1)(\al+1)$. We have $\ga=q\be=p(\al+1)$. Thus, 
\begin{eqnarray}\label{h-ep}
\vp(\ep) \int_{M}  \frac{\psi^p h'(r)}{{h(r)}^{\ga}} + p\ga\ep^p \int_{M} \frac{\psi^p h'(r)|\bar\na r^\perp|^2}{{h(r)}^{\ga}} &\le& \  \hspace{-0.3cm} \int_M [\frac{p^2|\na^M\psi|^2}{{h(r)}^{2\al}}+\frac{\psi^2|{H}|^2}{{h(r)}^{2\al}}]^{p/2}\nonumber  
\\&& + \ p\ep^p \int_{\p M} \frac{\psi^p}{{h(r)}^{\ga-1}}\lan \bar\na r,\nu \ran
\end{eqnarray}
where $\vp(\ep)=p\ep^p[(k-\ga)-\frac{\ep^q}{q h'(r_0)}] = \frac{p\ep^p}{h'(r_0)}[(k-\ga)h'(r_0)-\frac{\ep^q}{q}]$. Now, notice that
\begin{eqnarray*}
\frac{h'(r_0)}{p}\vp'(\ep)(\ep) &=& p\ep^{p-1}[(k-\ga)h'(r_0)-\frac{\ep^q}{q}] - \ep^p \ep^{q-1} 
\\&=& p\ep^{p-1}[(k-\ga)h'(r_0)-\frac{\ep^q}{q} - \frac{\ep}{p}\ep^{q-1}] 
%\\&=& p\ep^{p-1}[(k-\ga)h'(r_0)-(\frac{1}{q}+\frac{1}{p})e^q] 
\\&=& p\ep^{p-1}[(k-\ga)h'(r_0)-\ep^q]. 
\end{eqnarray*}
And, $\frac{h'(r_0)}{p^2}\vp''(\ep)=(p-1)\ep^{p-2}[(k-\ga)h'(r_0)-\ep^q]-q\ep^{p-1}\ep^{q-1}$. Thus, $\vp'(\ep)=0$ if and only if $\ep^q=(k-\ga)h'(r_0)>0$. At this point,  $\vp''(\ep)=-p^2q\ep^{p+q-2}< 0$. Hence, $\vp(\ep)$ attains its maximum at $\ep_0= [(k-\ga)h'(r_0)]^{\frac{p-1}{p}}$, with $\vp(\ep_0)=\frac{p}{h'(r_0)}[(k-\ga)h'(r_0)]^{p-1}(1-\frac{1}{q})(k-\ga)h'(r_0)=(k-\ga)^p h'(r_0)^{p-1}$.

Since $p\al=\ga-p$, by (\ref{h-ep}), and multiplying both sides by $\frac{1}{p^p}$, it holds
\begin{eqnarray*}
\frac{(k-\ga)^p h'(r_0)^{p-1}}{p^p} \int_{M} \frac{h'(r)\psi^p}{{h(r)}^{\ga}} + \frac{\ga[(k-\ga)h'(r_0)]^{p-1}}{p^{p-1}}\int_{M} \frac{\psi^ph'(r)|\bar\na r^\perp|^2}{{h(r)}^\ga} 
\nonumber\\ && \hspace{-12.5cm} \le  \int_M \frac{1}{{h(r)}^{\ga-p}}[|\na^M \psi|^2 + \frac{\psi^2|H|^2}{p^2}]^{p/2} + \frac{[(k-\ga)h'(r_0)]^{p-1}}{p^{p-1}}\int_{\p M} \frac{\psi^p}{{h(r)}^{\ga-1}}\lan \bar\na r,\nu \ran.  
\end{eqnarray*}

Now, we assume $\xi_0\in M$. Let $\m Z_0=\{x\in M \mid f(x)=\xi_0\}$. Since every immersion is locally an embedding, it follows that $\m Z_0$ is discrete, hence it is finite, since $M$ is compact. We write $\m Z_0=\{p_1,\ldots,p_l\}$ and let $\rho=r\circ f=d_{\bar M}(f\,,\xi_0)$. 
By a Nash Theorem, there is an isometric embedding of $\bar M$ in an Euclidean space $\real^N$. The composition of such immersion with $f$ induces an isometric immersion $\bar f:M\to \real^N$. By the compactness of $M$, finiteness of $\m Z_0$, and the local form of an immersion,  one can choose a small $\ep>0$, such that $ [\rho<2\ep]:=\rho^{-1}[0,2\ep) = U_1 \sqcup \ldots \sqcup U_l$ (disjoint union), where each $U_i$ is a neighborhood of $p_i$ in $M$ such that the restriction $\bar f|_{U_i}:U_i\to \real^N$ is a graph over a smooth function, say $u_i:U_i\to \real^{N-k}$. Thus, considering the set $[\rho<\de]=[r<\de]\cap M$ (identifying $x\in M$ with $f(x)$), with $0<\de<\ep$, again by the finiteness of $\m Z_0$, we have that  the volume $\vol_M([\rho<\de])=O(\de^k)$, as $\de\to 0$. Similarly, one also obtain that $\vol_{\p M}(\p M \cap [\rho<\de])=O(\de^{k-1})$.

Now, for each $0<\de<\ep$, consider the cut-off function $\eta=\eta_\de\in C^\infty(M)$  satisfying:
\begin{eqnarray}\label{cut-off}
&& 0\le \eta \le 1, \mbox{ in } M; \nonumber\\
&&\eta = 0, \mbox{ in } [\,\rho<\de], \ \mbox{ and } \eta = 1, \mbox{ in } [\,\rho>2\de];\\
&& |\na^M \eta|\le L/\de, \nonumber
\end{eqnarray}
for some constant $L>1$, that does not depend on $\de$ and $\eta$. Consider $\phi=\eta\psi$. Since $\phi\in C^1(M)$ and   $\xi_0\notin M':=\spp(\phi)$, it holds 
\begin{eqnarray*}\label{hardy-0-phi}
\frac{(k-\ga)^p h'(r_0)^{p-1}}{p^p} \int_{M} \frac{\phi^ph'(r) }{{h(r)}^{\ga}} + \frac{\ga[(k-\ga)h'(r_0)]^{p-1}}{p^{p-1} }\int_{M} \frac{\phi^p h'(r) |\bar\na r^\perp|^2}{{h(r)}^{\ga}} 
\nonumber\\ && \hspace{-12.5cm} \le \ \int_M \frac{h'(r_0)^{1-p}}{{h(r)}^{\ga-p}}[|\na^M \phi|^2 + \frac{\phi^2| H|^2}{p^2}]^{p/2} + \frac{[(k-\ga)h'(r_0)]^{p-1}}{p^{p-1}} \int_{\p M} \frac{\phi^p}{{h(r)}^{\ga-1}}\lan \bar\na r,\nu \ran.
\end{eqnarray*}

The integral $\int_{\p M} \frac{\phi^p}{{h(r)}^{\ga-1}}\lan \bar\na r,\nu \ran$ exists, since  $\int_{\p M} \frac{\psi^p}{{h(r)}^{\ga-1}}|\lan \bar\na r,\nu\ran|$ exists and $0\le \phi\le \psi$.   Furthermore, 
\begin{eqnarray}\label{est-grad-0-phi-1}
\int_M \frac{1}{{h(r)}^{\ga-p}}[|\na^M\phi|^2+\frac{\phi^2|{H}|^2}{p^2}]^{p/2} &=& \int_{[\rho>2\de]} \frac{1}{{h(r)}^{\ga-p}}[|\na^M\psi|^2+\frac{\psi^2|{H}|^2}{p^2}]^{p/2} 
\nonumber \\&& \hspace{-0.6cm}+ \int_{[\de<\rho<2\de]} \frac{1}{{h(r)}^{\ga-p}}[|\na^M\phi|^2+\frac{\phi^2|{H}|^2}{p^2}]^{p/2} 
\end{eqnarray}
and
\begin{eqnarray*}
\int_{[\de<\rho<2\de]} \frac{1}{{h(r)}^{\ga-p}}[|\na^M\phi|^2+\frac{\phi^2|{H}|^2}{p^2}]^{p/2} && 
\\&& \hspace{-5cm} \le \ \int_{[\de<\rho<2\de]} \frac{O(1)}{{h(r)}^{\ga-p}}[\eta^p|\na^M\psi|^p +|\psi|^p |\na^M\eta|^p + |\psi|^p {|H|^p}]
\\&& \hspace{-5cm} = \ \int_{[\de<\rho<2\de]} O(\frac{1}{h(\de)^{\ga-p}}) (O(1)+ O(\frac{1}{\de^p}))
\\&& \hspace{-5cm} = \ O(\frac{1}{\de^{\ga-p}}) (O(1)+ O(\frac{1}{\de^p}))O(\de^k) = O(\de^{k-\ga}),
\end{eqnarray*}
as $\de\to 0$. Therefore, it holds  
\begin{eqnarray*}
\frac{(k-\ga)^ph'(r_0)^{p-1}}{p^p } \int_{[r>2\de]\cap M} \frac{\psi^ph'(r)}{{h(r)}^{\ga}} 
\\&& \hspace{-3cm}+ \frac{\ga[(k-\ga)h'(r_0)]^{p-1}}{p^{p-1} }\int_{[r>2\de]\cap M} \frac{\psi^ph'(r)|\bar\na r^\perp|^2}{{h(r)}^{\ga}} 
\nonumber\\ && \hspace{-6cm} \le \ \int_M \frac{1}{{h(r)}^{\ga-p}}[|\na^M \psi|^2 + \frac{\psi^2|H|^2}{p^2}]^{p/2} 
\\&&\hspace{-5cm} + \ \frac{[(k-\ga)h'(r_0)]^{p-1}}{p^{p-1}} \int_{\p M \cap [\rho>2\de]} \frac{\psi^p}{{h(r)}^{\ga-1}}\lan \bar\na r,\nu \ran + O(\de^{k-\ga}).
\end{eqnarray*}
Proposition \ref{teo-hardy-norm}, follows, since $k-\ga>0$ and $\int_{\p M} \frac{\psi^p}{{h(r)}^{\ga-1}}\lan \bar\na r,\nu \ran$ exits.
\end{proof}

It is simple to see that, for all numbers $a\ge 0$ and $b\ge 0$, it holds
\begin{equation}\label{num-property}
\min\{1,2^{\frac{p-2}{2}}\}(a^p+b^p) \le (a^2+b^2)^{p/2} \le \max\{1,2^{\frac{p-2}{2}}\}(a^p+b^p).
\end{equation}
In fact, to show this, without loss of generality, we can suppose $a^2+b^2=1$. We write $a=\cos\te$ and $b=\sin\te$, for some $\te\in [0,\pi/2]$. If $p=2$, there is nothing to do. Assume $p\neq 2$. Consider $f(\te)= a^p+b^p=cos^p(\te)+\sin^p(\te)$. 

The derivative of $f$ is given by $f'(\te)=-p\cos^{p-1}\sin(\te)+p\sin^{p-1}\cos(\te)$. Thus, $f'(\te)=0$ iff  $\cos^{p-1}\sin(\te)=\sin^{p-1}\cos(\te)=0$, that is, iff either $\cos(\te)=0$ or $\sin(\te)=0$, or $\cos^{p-2}(\te)=\sin^{p-2}(\te)$. Thus, $f'(\te)=0$ iff $\te=0$, $\te=\frac{\pi}{2}$, or $\te=\frac{\pi}{4}$.  So, the critical values are $f(0)=f(\frac{\pi}{2})=1$ and $f(\frac{\pi}{4})=2 (\frac{1}{\sqrt{2}})^{p}=2^{1-\frac{p}{2}}$. 
Thus, $\min\{1,2^{1-\frac{p}{2}}\} \le f(\te)=a^p+b^p\le \max\{1,2^{1-\frac{p}{2}}\}$. 
Hence, (\ref{num-property}) follows. 
\\

As a consequence of (\ref{num-property}) and Proposition \ref{teo-hardy-norm}, we obtain the following Hardy inequality.

\begin{theorem}\label{teo-hardy} Fixed $\xi_0\in \bar M$, we assume $(\bar K_{\rad})_{\xi_0}\le \m K(r)$, where $r=r_{\xi_0}=d_{\bar M}(\cdot\,,\xi_0)$. Assume that $M$ is contained in a ball $\m B=\m B_{r_0}(\xi_0)$, for some $0<r_0<\min\{\bar r_0(\m K),\, Inj_{\bar M}(\xi_0)\}$.  Let $1\le p<\infty$ and $-\infty<\ga<k$. Then, for all $\psi\in C^1(M)$, it holds
\begin{eqnarray*}
\frac{(k-\ga)^p h'(r_0)^{p-1}}{p^p} \int_{M} \frac{|\psi|^p h'(r)}{h(r)^{\ga}} + \frac{\ga[(k-\ga)h'(r_0)]^{p-1}}{p^{p-1}}\int_M \frac{|\psi|^p h'(r)|\bar\na r^\perp|^2}{h(r)^{\ga}} 
\\&& \hspace{-11.5cm}
\le  \ \m A_p\int_M [\frac{|\na^M \psi|^p}{h(r)^{\ga-p}} + \frac{|\psi|^p|H|^p}{p^p h(r)^{\ga-p}}] +\frac{[(k-\ga)h'(r_0)]^{p-1}}{p^{p-1}} \int_{\p M} \frac{|\psi|^p}{h(r)^{\ga-1}}.
\end{eqnarray*}
where $\m A_p=\max\{1, 2^{\frac{p-2}{2}}\}$. Moreover, if $M$ is minimal, we can take  $\m A=1$. 
\end{theorem} 

\begin{proof} We may assume $h'(r_0)>0$, otherwise, there is nothing to do.  First, we fix $p>1$ and let $\psi\in C^1(M)$. Take $\ep>0$ and consider the function $\psi_{\ep}=(\psi^2+\ep^2)^{1/2}$. Note that $\psi_\ep\ge |\psi|\ge 0$ and $|\na \psi_\ep| = \frac{\psi}{(\psi^2+\ep^2)^{{1}/{2}}}|\na \psi|\le |\na\psi|$. Thus, by Proposition \ref{teo-hardy-norm}, 
\begin{eqnarray*}
\frac{(k-\ga)^p }{p^p} \int_{M} \frac{\psi_\ep^p h'(r)}{{h(r)}^{\ga}} + \frac{\ga(k-\ga)^{p-1}}{p^{p-1}}\int_M \frac{\psi_\ep^p h'(r)|\bar\na r^\perp|^2}{{h(r)}^{\ga}} 
\\&& \hspace{-8cm}\le
\int_M \frac{h'(r_0)^{1-p}}{{h(r)}^{\ga-p}}[|\na^M \psi|^2 + \frac{\psi_\ep^2|H|^2}{p^2}]^{{p}/{2}} + \frac{(k-\ga)^{p-1}}{p^{p-1}} \int_{\p M} \frac{\psi_\ep^p}{{h(r)}^{\ga-1}}.
\end{eqnarray*}
Since $\psi_{\ep_1}\le \psi_{\ep_2}$, if $\ep_1<\ep_2$, and $|\psi|\le \psi_\ep \le |\psi|+\ep$, by taking $\ep\to 0$, we have
\begin{eqnarray}\label{hardy-ineq-norm-p>1}
\frac{(k-\ga)^p}{p^p} \int_{M} \frac{|\psi|^p h'(r)}{{h(r)}^{\ga}} + \frac{\ga(k-\ga)^{p-1}}{p^{p-1}}\int_M \frac{|\psi|^ph'(r)|\bar\na r^\perp|^2}{{h(r)}^{\ga}} 
\nonumber\\&& \hspace{-8.5cm}\le
\int_M \frac{h'(r_0)^{1-p}}{{h(r)}^{\ga-p}}[|\na^M \psi|^2 + \frac{\psi^2|H|^2}{p^2}]^{{p}/{2}} + \frac{(k-\ga)^{p-1}}{p^{p-1}} \int_{\p M} \frac{|\psi|^p}{{h(r)}^{\ga-1}}.
\end{eqnarray}
Now, taking $p\to 1$, and applying the dominated convergence theorem, we obtain that (\ref{hardy-ineq-norm-p>1}) also holds for $p=1$. Applying (\ref{num-property}) in inequality (\ref{hardy-ineq-norm-p>1}), Theorem \ref{teo-hardy} follows.
\end{proof}

 As a consequence of Theorem \ref{teo-hardy}, we obtain a Hardy type inequality for submanifolds in ambient spaces having a pole with nonpositive radial sectional curvature. Namely, the following result holds

\begin{corollary}\label{hardy-hadamard-cor} Let $\bar M$ be a complete simply-connected manifold with radial sectional curvature $(\bar K_{\rad})_{\xi_0}\le 0$, for some $\xi_0\in \bar M$. 
Let $r=r_{\xi_0}=d_{\bar M}(\cdot\,,\xi_0)$ and let $1\le p <k$ and $-\infty<\ga<k$. Then, for all $\psi\in C^1(M)$, it holds
\begin{eqnarray*}
\frac{(k-\ga)^p}{p^p} \int_{M} \frac{|\psi|^p}{{r}^{\ga}} + \frac{\ga(k-\ga)^{p-1}}{p^{p-1}}\int_M \frac{|\psi|^p|\bar\na r^\perp|^2}{{r}^{\ga}} 
\\&& \hspace{-5.5cm}
\le  \ \m A_p\int_M [\frac{|\na^M \psi|^p}{{r}^{\ga-p}} + \frac{|\psi|^p|H|^p}{p^p {r}^{\ga-p}}] +\frac{(k-\ga)^{p-1}}{p^{p-1}} \int_{\p M} \frac{|\psi|^p}{{r}^{\ga-1}}.
\end{eqnarray*}
where $\m A_p=\max\{1, 2^{\frac{p-2}{2}}\}$. Moreover, if $M$ is minimal, we can take  $\m A_p=1$. 
\end{corollary}

%As a particular case, we assume that $X(M)\subset \Sp_c^{n-1}=\{v\in \real^n \mid |v|=c\}$, with $c>0$. Then, it holds $X^\perp = X$ everywhere in $M$, since $X(x)$ is normal to $\Sp_c^{n-1}$. Thus, as a consequence of Theorem \ref{teo-hardy}, we obtain the Hardy inequality on submanifolds  in Euclidean spheres. \begin{corollary} Let $X:M^k\to \Sp_c^{n-1}$, with $c>0$, be an isometric immersion of a compact manifold $M$ with (possibly nonempty) boundary $\p M$ in an Euclidean hypersphere $\Sp_c^{n-1}$. Let $1\le p < \infty$ and $-\infty<\ga<k$. For all $\psi\in C^1(M)$, it holds \begin{eqnarray*} \frac{(k-\ga)^{p-1}[k+p(\ga-1)]}{p^{p}} \int_{M} \frac{|\psi|^p}{|X|^{\ga}}  \\&& \hspace{-4.0cm} \le  \ \m A_p \int_M [\frac{|\na^M \psi|^p}{|X|^{\ga-p}} + \frac{|\psi|^p|\vec H|^p}{p^p |X|^{\ga-p}}] +\frac{(k-\ga)^{p-1}}{p^{p-1}} \int_{\p M} \frac{|\psi|^p}{|X|^{\ga-1}}. \end{eqnarray*} where $\m A_p=\max\{1, 2^{\frac{p-2}{2}}\}$. Moreover, if $M$ is minimal, we can take  $\m A=1$.  \end{corollary}

\section{The weighted Hoffman-Spruck inequality for submanifolds}

Another consequence of Theorem \ref{teo-hardy} is the Hoffman-Spruck Inequality. Namely, fixed $\xi_0\in \bar M$, we assume $(\bar K)_{\rad}\le \m K(r)$ in $\bar M$,  where $r=r_{\xi_0}=d_{\bar M}(\cdot\,,\xi_0)$. Let $\m B$ be the geodesic ball in $\bar M$ centered at $\xi_0$ and radius $\bar r_0=\min\{\bar r_0(\m K),Inj_{\bar M}(\xi_0)\}$.  Since $M$ is compact and contained in $\m B$, it follows that $r^*=\max_{x\in M} r(x)<\bar r_0= \min\{\bar r_0(\m K), Inj_{\bar M}(\xi_0)\}$ and $h$ is increasing in $[0,\bar r_0)$. Hence, we may assume $M$ is contained in a ball $\m B_{r_0}(\xi_0)$, for some $0<r_0<\min\{\bar r_0(\m K),Inj_{\bar M}(\xi_0)\}$, , arbitrarily close to $\bar r_0$, satisfying $h'(r_0)>0$. In particular, $h'(r)>0$, for all $0\le r\le r_0$, since $h''=-\m K h\le 0$. Applying Theorem \ref{teo-hardy}, we conclude  $M$ cannot be minimal. On the other hand, notice that  $\eta=-\bar \na r$ is the unit normal vector to the boundary $\p \m B_{r_0}(\xi_0)$ pointing inward $\m B_{r_0}(\xi_0)$. By the Hessian comparison theorem (see (\ref{hes-comp})), the shape operator $A=-\bar\na\eta$ satisfies $A(v,v)=\hs_{r}(v,v)\ge \frac{h'(r_0)}{h(r_0)}> 0$, for all unit vector $v$ tangent to $\p \m B_{r_0}(\xi_0)$. Hence, the boundary $\p \m B_{r_0}(\xi_0)$ is convex. Since $\bar M$ does not admit any closed minimal submanifold inside $\m B$, by Section 6 of \cite{W}, there exists a constant $c>0$ depending only on $k$, satisfying 
\begin{equation*}
\vol(M)^{\frac{n-1}{n}}\le c \big(\vol(\p M)+\int_M |H|\big), 
\end{equation*} 
provided $n<7$, or $\vol(M)<\m D$, where $\m D$ depends only on $\m B$. Thus, by a straightforward calculus, one obtain 
\begin{equation} \label{HS-White-version}
[\int_M |\psi|^{p^*}]^\frac{p}{p^*} \le S \int_M (|\na^M \psi|^p + \frac{|\psi|^p|{H}|^p}{p^p}), 
\end{equation}
for all $1\le p<k$ and $\psi\in C^1(M)$, with $\psi=0$ on $\p M$, where $S$ depends only on $k$ and $p$.

By Hoffman and Spruck \cite {HS}, one see that $\m D$ depends only on $Inj_{\bar M}(M)$ and $\bar r_0(\m K)$. Namely, Hoffman and Spruck proved  the following.
\begin{theorem}\label{hoffman-spruck-teo} 
Assume  the sectional curvatures of $\bar M$ satisfy $\bar K \le b^2$,  for some constant $b\ge 0$. Then, there exists a constant $S>0$ satisfying
\begin{equation}\label{desig-hoff-sprk}
[\int_M |\psi|^{p^*}]^\frac{p}{p^*} \le S \int_M (|\na^M \psi|^p + \frac{|\psi|^p|{H}|^p}{p^p}), 
\end{equation}
for all $1\le p<k$ and $\psi\in C^1(M)$, with $\psi=0$ on $\p M$, provided there exists $z\in (0,1)$ satisfying
\begin{eqnarray}
\label{bar J} && \bar J_z:= [\frac{\om_k^{-1}}{1-z}\vol_M(\spp(\psi))]^{\frac{1}{k}}< \frac{1}{b}, \mbox{ if } b>0; \mbox{ and } \\
\label{h-bar J} && 2h^{-1}_b(\bar J_z)\le Inj_{\bar M}\,(\spp(\psi)),
\end{eqnarray}
where $h_b(t)=t$, with $t\in (0,\infty)$, if $b=0$, and 
$h_b(t)=\frac{1}{b}\sin(bt)$, with $t\in (0,\frac{\pi}{2b})$, if $b>0$ (In this case, $h^{-1}_b(t)=\frac{1}{b}\sin^{-1}(tb)$, with $t\in (0,\frac{1}{b})$).
Here, $\om_k$ is the volume of the standard unit ball $B_1(0)$ in $\real^k$, and $\mbox{Inj}\,(\spp(\psi))$ is the infimum of the injectivity radius of $\bar M$ restricted to the points of $\spp(\psi)$. Furthermore, the constant $S=S_{k,z}$ is given by
\begin{equation}\label{hoff-sprk-S}
S_{k,p,z}=\frac{\pi}{2}\, \frac{2^{k} k}{z(k-1)}\big(\frac{\om_k^{-1}}{1-z}\big)^{\frac{1}{k}}\,2^{p-1}[\frac{p(k-1)}{k-p}]^p.
\end{equation}
Moreover, if $b=0$, $S_{k,p,z}$ can be improved by taking $1$ instead $\frac{\pi}{2}$.
\end{theorem}
\begin{remark} The Hoffman-Spruck's Theorem above can be generalized for ambient spaces $\bar M$ satisfying $(\bar K_{\rad})_\xi \le \m K(r_\xi)$, for all $\xi\in \bar M$. The details and proof for this case can be found, for instance, in \cite{BM}.
\end{remark}

The constant $S_{k,p,z}$ as in (\ref{hoff-sprk-S}) reaches its minimum at $z=\frac{k}{k+1}$, hence we can take
\begin{eqnarray}\label{constant-S}
S &=& S_{k,p}=\min_{z\in (0,1)} S_{k,p,z} = \frac{\pi}{2} \, \frac{2^{k} k}{\frac{k}{k+1}(k-1)}\big(\om_k^{-1}(k+1)\big)^{\frac{1}{k}} \,2^{p-1}[\frac{p(k-1)}{k-p}]^p
\nonumber\\&=&  \frac{\pi}{2} \, \frac{2^{k}(k+1)^\frac{k+1}{k}}{k-1} \om_k^{-\frac{1}{k}}\,2^{p-1}[\frac{p(k-1)}{k-p}]^p,
\end{eqnarray}
provided  $\bar J= [\frac{k+1}{\om_k}\vol_M(\spp(\psi))]^\frac{1}{k}\le s_b$ and $2h^{-1}(\bar J)\le  Inj_{\bar M}(\spp(\psi))$. 
Thus, as a corollary of Theorem \ref{hoffman-spruck-teo} and (\ref{HS-White-version}), one has

\begin{proposition}\label{HS-compact} Fixed  $\xi_0\in \bar M$, assume $(\bar K_{\rad})_{\xi_0}\le \m K(r_{\xi_0})$. 
Assume $M$ is contained in $\m B=\m B_{r_0}(\xi_0)$, with $r_0=\min\{\bar r_0(\m K), Inj_{\bar M}(\xi_0)\}$.   Then, for all $1\le p<k$, there exists $S>0$, depending only on $k$ and $p$, such that, for all $\psi\in C^1(M)$ with $\psi=0$ on $\p M$, it holds
\begin{equation*}
[\int_M |\psi|^{p^*}]^\frac{p}{p^*} \le S \int_M (|\na^M \psi|^p + \frac{|\psi|^p|{H}|^p}{p^p}), 
\end{equation*}
provided either $k<7$, or $\vol_M(\spp(\psi))<\m D$, where $0<\m D\le +\infty$, depends only on $Inj_{\bar M}(M)$ and $r_0$.
\end{proposition}
%\begin{proof} Since $\bar M$ is complete, we can take $b\ge 0$ such that $\bar K\le b^2$, in $\bar B_{r_0}$, and $\m A=\inf_{x\in \m B_{r_0}} Inj_{\bar M}(x)>0$.\end{proof}

%\begin{corollary}\label{cor-HS-bounded geometry} Assume $\bar M$ has bounded geometry, i.e. $\bar M$ has sectional curvature $\bar K\le b^2$ and injectivity radius $Inj_{\bar M}\ge \de>0$. There exists a positive constant $S=S_{k,\de}>0$ depending only $k$ and $\de$, such that if $\vol(M)<\La_{k,\de}$ then, for all $\psi\in C^1(M)$, with $\psi=0$ on $\p M$, one has \begin{equation*} [\int_M |\psi|^{p^*}]^\frac{p}{p^*} \le S\int_M (|\na^M \psi|^p + |\psi|^p|{H}|^p), \end{equation*} where $S=S_{k,p}=2^{k+p-1} \frac{\pi}{2} \frac{(k+1)^\frac{k+1}{k}}{k-1} \om_k^{-\frac{1}{k}}$. Moreover, if $b=0$ one can improve $S$ by taking $1$ instead $\frac{\pi}{2}$. \end{corollary}

\begin{example} \label{cartan-hadamard-example}
Assume $\bar M$ is a Cartan-Hadamard manifold. Then, it holds $\bar r_0(\m K)=Inj_{\bar M}(\xi_0)=Inj_{\bar M}\,(M)=\infty$. Hence, we can take $\m D=+\infty$ in Proposition \ref{HS-compact}.
\end{example}

Now, we use Theorem \ref{teo-hardy} together with Theorem \ref{hoffman-spruck-teo}, in order to obtain a weighted Hoffman-Spruck inequality for submanifolds in  manifolds. 

\begin{theorem}\label{teo-sob} Fixed  $\xi_0\in \bar M$, assume $(\bar K_{\rad})_{\xi_0}\le \m K(r)$, where $r=d_{\bar M}(\cdot\,,\xi_0)$. 
Assume $M$ is contained in $\m B=\m B_{r_0}(\xi_0)$, with $r_0=\min\{\bar r_0(\m K), Inj_{\bar M}(\xi_0)\}$.  Then, for all $\psi\in C^1(M)$, with $\psi=0$ on $\p M$, it holds
\begin{eqnarray*}
\frac{1}{S}[\int_M\frac{|\psi|^{p^*}}{{h(r)}^{p^*\al}}]^{p/p^*} + \Phi_{k,p,\al}\int_M \frac{|\psi|^p{h'(r)}|\bar\na r^\perp|^2}{{h(r)}^{p(\al+1)}} +  \De_{k,p,\al} \int_M \frac{|\psi|^p{h'(r)}|\bar\na r^\perp|^p}{{h(r)}^{p(\al+1)}} 
\\&& \hspace{-6.2cm} \le \Ga_{k,p,\al} \int_M[\frac{|\na^M\psi|^p}{{h(r)}^{p\al}}+ \frac{|\psi|^p|H|^p}{p^p{h(r)}^{p\al}}],
\end{eqnarray*}
provided either $k<7$, or $\vol(M)<\m D$, where $0<\m D\le +\infty$ depends only $inj_{\bar M}(M)$ and $r_0$.  Here, $p^*=\frac{kp}{k-p}$, $S>0$ depends only on $k$ and $p$, and
\begin{eqnarray*}
\Ga_{k,p,\al}&=& \m A_p \big[1+|\al|^{\frac{2p}{2+p}}2^\frac{|p-2|}{(p+2)}{h'(r_0)}^{\frac{2(1-p)}{2+p}}(\frac{p}{k-\ga})^\frac{2p}{p+2}\,\big]^{\frac{p+2}{2}}\\
&=& h'(r_0)^{1-p}\m A_p \big[h'(r_0)^\frac{2(p-1)}{p+2}+|\al|^{\frac{2p}{2+p}}2^\frac{|p-2|}{(p+2)}(\frac{p}{k-\ga})^\frac{2p}{p+2}\,\big]^{\frac{p+2}{2}}\\
\Phi_{k,p,\al}&=& 2^{\frac{|p-2|}{2}}\frac{\ga p}{k-\ga} (|\al|^{\frac{2p}{2+p}}+2^{\frac{-|p-2|}{p+2}}h'(r_0)^{\frac{2(p-1)}{p+2}}(\frac{p}{k-\ga})^{\frac{-2p}{p+2}})^{\frac{p}{2}}|\al|^{\frac{2p}{2+p}}\\
\De_{k,p,\al}&=& \m A_p(|\al|^{\frac{2p}{2+p}}+2^{\frac{-|p-2|}{p+2}}h'(r_0)^{\frac{2(p-1)}{p+2}}(\frac{p}{k-\ga})^{\frac{-2p}{p+2}})^{\frac{p}{2}}|\al|^{\frac{2p}{2+p}},
\end{eqnarray*}
where $\ga=p(\al+1)$ and $\m A_p=\max\{1,2^{\frac{p-2}{2}}\}$.
\end{theorem}

\begin{proof}  First, we assume $\xi_0\notin M$. 
Then,  $r=d_{\bar M}(\cdot\,,\xi_0)>0$ on $M$, hence  $\frac{\psi}{{h(r)}^\al}$ is a $C^1$ function on $M$ vanishing on $\p M$. By Proposition \ref{HS-compact}, there is a constant $S>0$, depending only on $k$ and $p$, such that, for all $\psi\in C^1(M)$ with $\psi=0$ on $\p M$, the following inequality holds 
\begin{equation}\label{Michael-Simon-psi}
[\int_M \frac{|\psi|^{p^*}}{{h(r)}^{p^*\al}}]^{\frac{p}{p^*}}\le S\int_M [|\na^M(\frac{\psi}{{h(r)}^\al})|^p + \frac{|\psi|^p|H|^p}{p^p{h(r)}^{p\al}}], 
\end{equation}
provided $k<7$ or $\vol(M)\le \m D$, where  $\m D$, depends only on $r_0$ and $Inj_{\bar M}(M)$.

Using that $\na^M (\frac{\psi}{{h(r)}^\al}) = \frac{\na^M\psi}{{h(r)}^\al} - \frac{\al\psi h'(r)\bar\na r^\top}{{h(r)}^{\al+1}}$,  by the Young inequality, 
\begin{eqnarray}
|\na^M(\frac{\psi}{{h(r)}^\al})|^2 &=& \frac{\al^2\psi^2 h'(r)^2|\bar\na r^\top|^2}{{h(r)}^{2\al+2}} + (\frac{-\al}{{h(r)}^{2\al}}) ~ 2\lan\na^M\psi, \frac{\psi h'(r)\bar\na r^\top}{{h(r)}}\ran
\nonumber\\&& + \ \frac{|\na^M\psi|^2}{{h(r)}^{2\al}} 
\nonumber\\ &\le& (\al^2+|\al|\ep^2)\frac{\psi^2 h'(r)^2|\bar\na r^\top|^2}{{h(r)}^{2\al+2}} + (1+\frac{|\al|}{\ep^2})\frac{|\na^M\psi|^2}{{h(r)}^{2\al}}
%\nonumber\\ &=& (\al^2+|\al|\ep^2)\frac{\psi^2 h'(r)^2|\bar\na r^\top|^2}{{h(r)}^{2(\al+1)}} + (1+\frac{|\al|}{\ep^2})\frac{|\na^M\psi|^2}{{h(r)}^{2\al}}
\nonumber\\ &=& (|\al|+\ep^2)\big[\frac{|\al|\psi^2 h'(r)^2|\bar\na r^\top|^2}{{h(r)}^{2(\al+1)}}+ \frac{|\na^M\psi|^2}{\ep^2 {h(r)}^{2\al}}\big],
\end{eqnarray}
for all $\ep>0$. Hence, using (\ref{num-property}), 
\begin{eqnarray}\label{a0}
|\na^\Si(\frac{\psi}{{h(r)}^\al})|^p &=& (|\na^M(\frac{\psi}{{h(r)}^\al})|^2)^{\frac{p}{2}} 
%\nonumber\\ &\le& \m A_p (|\al|+\ep^2)^\frac{p}{2}[\frac{|\al|^\frac{p}{2} \psi^p {h'(r)}^p|\bar\na r^\top|^p}{{h(r)}^{p(\al+1)}} + \frac{|\na^M\psi|^p}{\ep^p {h(r)}^{p\al}}]
\nonumber\\ &\le& \m A_p (|\al|+\ep^2)^\frac{p}{2}[\frac{|\al|^\frac{p}{2} \psi^p {h'(r)}|\bar\na r^\top|^p}{{h(r)}^{p(\al+1)}} + \frac{|\na^M\psi|^p}{\ep^p {h(r)}^{p\al}}]
\label{h'-without-p}\\&\le& \m A_p (|\al|+\ep^2)^\frac{p}{2}[\frac{|\al|^\frac{p}{2} \psi^p {h'(r)}}{{h(r)}^{p(\al+1)}}(\m B_p  - |\bar\na r^\perp|^p) + \frac{|\na^M\psi|^p}{\ep^p {h(r)}^{p\al}}]
\\&& \hspace{-2cm}= \m A_p (|\al|+\ep^2)^\frac{p}{2}[\frac{(\m B_p |\al|^\frac{p}{2} \psi^p {h'(r)}}{{h(r)}^{p(\al+1)}} - \frac{|\al|^\frac{p}{2} \psi^p {h'(r)} |\bar\na r^\perp|^p)}{{h(r)}^{p(\al+1)}} + \frac{|\na^M\psi|^p}{\ep^p {h(r)}^{p\al}}]. \nonumber
\end{eqnarray}
where $\m A_p = \max\{1,2^\frac{p-2}{2}\}$ and $\m B_p=\max\{1,2^{\frac{2-p}{2}}\}$. 
Inequality (\ref{h'-without-p}) holds since $h''\le 0$, hence $h'(r)\le h'(0)=1$ and Inequality (\ref{a0}) holds since, by (\ref{num-property}), one has $|\bar\na r^\top|^p+|\bar\na r^\perp|^p\le \max\{1,2^{\frac{2-p}{2}}\}$. Thus, using $(\ref{Michael-Simon-psi})$ and $(\ref{a0})$,  we obtain
\begin{eqnarray*}
\frac{1}{S}[\int_M\frac{\psi^{p^*}}{{h(r)}^{p^*\al}}]^{p/p^*} &\le &  \int_M \frac{\psi^p|H|^p}{p^p{h(r)}^{p\al}} \ + \ 
\\&& \hspace{-4.2cm} \m A_p(|\al|+\ep^2)^{\frac{p}{2}} \big[\,|\al|^{\frac{p}{2}}\m B_p \int_M \frac{\psi^p {h'(r)}}{{h(r)}^{p(\al+1)}} + \frac{1}{\ep^p}
 \int_M \frac{|\na^M\psi|^p}{{h(r)}^{p\al}} - |\al|^{\frac{p}{2}} \int_M \frac{\psi^p {h'(r)} |\bar\na r^\perp|^p}{{h(r)}^{p(\al+1)}}\big]. 
\end{eqnarray*}

On the other hand, by using Theorem \ref{teo-hardy}, 
\begin{equation*}\label{a1}
\int_M~\frac{\psi^p {h'(r)}}{{h(r)}^{p(\al+1)}} ~ \le A_{k,p,\al} \int_M [\frac{|\na^M \psi|^p}{{h(r)}^{p\al}}+\frac{\psi^p|H|^p}{p^p{h(r)}^{p\al}}] - B_{k,p,\al} \int_M \frac{\psi^p {h'(r)}|\bar\na r^\perp|^2}{{h(r)}^{p(\al+1)}},
\end{equation*}
where $A_{k,p,\al} = \frac{\m A_p}{h'(r_0)^{p-1}}\frac{p^p}{(k-\ga)^p}$ and $B_{k,p,\al} =  \frac{p^p h'(r_0)^{1-p}}{(k-\ga)^p}\frac{\ga[(k-\ga)h'(r_0)]^{p-1}}{p^{p-1}}=\frac{\ga p}{k-\ga}$. 
Thus, it holds
\begin{eqnarray*}
\frac{1}{S}[\int_M\frac{\psi^{p^*}}{{h(r)}^{p^*\al}}]^{p/p^*} &\le& \ C_{k,\al,p,\ep} \int_M \frac{|\na^M \psi|^p}{{h(r)}^{p\al}} + D_{k,p,\al,\ep} \int_M \frac{\psi^p|H|^p}{p^p{h(r)}^{p\al}}\\
&&\ - \ E_{k,p,\al,\ep} \int_M \frac{\psi^p|\bar\na r^\perp|^2}{{h(r)}^{p(\al+1)}} - F_{k,p,\al,\ep}  \int_M \frac{\psi^p|\bar\na r^\perp|^p}{{h(r)}^{p(\al+1)}}
\nonumber,  
\end{eqnarray*}
where
\begin{eqnarray*}
C_{k,p,\al,\ep} &=& \m A_p(|\al|+\ep^2)^{\frac{p}{2}}(\,|\al|^{\frac{p}{2}}\m B_p A_{k,p,\al} + \ep^{-p}) \\&=& \m A_p[(|\al|+\ep^2)^{\frac{p}{2}}\,|\al|^{\frac{p}{2}}\m B_p A_{k,p,\al} + (1+ |\al|\ep^{-2})^{\frac{p}{2}}]
\\
D_{k,p,\al,\ep} &=& \m A_p(|\al|+\ep^2)^{\frac{p}{2}}|\al|^{\frac{p}{2}}\m B_p A_{k,p,\al} + 1 \le C_{k,p,\al,\ep}
\\
E_{k,p,\al,\ep} &=& \m A_p(|\al|+\ep^2)^{\frac{p}{2}}|\al|^{\frac{p}{2}} \m B_p B_{k,p,\al}.
\\
F_{k,p,\al,\ep} &=& \m A_p(|\al|+\ep^2)^{\frac{p}{2}}|\al|^{\frac{p}{2}}.
\end{eqnarray*}
Consider the function $k(\ep)= C_{k,p,\al,\ep}$. We have 
\begin{eqnarray*}
\m A_p^{-1}k'(\ep) &=& \frac{p}{2}(|\al|+\ep^2)^{\frac{p}{2}-1}\, 2\ep \,|\al|^{\frac{p}{2}}\m B_p A_{k,p,\al} + \frac{p}{2}(1+|\al|\ep^{-2})^{\frac{p}{2}-1}(-2|\al|\ep^{-3})
\\&=& p(|\al|+\ep^2)^{\frac{p}{2}-1}\, [\ep \,|\al|^{\frac{p}{2}}\m B_p A_{k,p,\al} - \ep^{2-p}|\al|\ep^{-3}]
\\&=& p(|\al|+\ep^2)^{\frac{p}{2}-1}\ep\,|\al| [\,|\al|^{\frac{p-2}{2}}\m B_p A_{k,p,\al} - \ep^{-2-p}].
\end{eqnarray*}
Thus, $k'(\ep)=0$ iff $\ep^{-2-p} = |\al|^{\frac{p-2}{2}}\m B_p A_{k,p,\al}$, i.e., $\ep=[|\al|^{\frac{p-2}{2}}\m B_p A_{k,p,\al}]^{\frac{-1}{p+2}}$. 
Hence, it simple to see  $k(\ep)$ reachs its minimum at $\ep_0=[|\al|^{\frac{p-2}{2}}\m B_p A_{k,p,\al}]^{\frac{-1}{p+2}}$. We obtain 
\begin{eqnarray*}
\Ga_{k,p,\al}&:=& C_{k,p,\al,\ep_0} = \m A_p(|\al|+\ep_0^2)^{\frac{p}{2}}(\,|\al|^{\frac{p}{2}}\m B_p A_{k,p,\al} + \ep_0^{-p})
\\&=& \m A_p(|\al|+[|\al|^\frac{p-2}{2}\m B_p A_{k,p,\al}]^{\frac{-2}{p+2}}\,)^{\frac{p}{2}}(\,|\al|^{\frac{p}{2}}\m B_p A_{k,p,\al} + [|\al|^{\frac{p-2}{2}}\m B_p A_{k,p,\al}]^{\frac{p}{p+2}})
\\&=& \m A_p|\al|^{\frac{p}{2}}(1+|\al|^{\frac{-2p}{2+p}}[\m B_p A_{k,p,\al}]^{\frac{-2}{p+2}}\,)^{\frac{p}{2}}(\,1 + |\al|^{\frac{-2p}{p+2}}[\m B_p A_{k,p,\al}]^{\frac{-2}{p+2}}) |\al|^{\frac{p}{2}}\m B_pA_{k,p,\al} 
\\&=& \m A_p|\al|^{\frac{p}{2}}(1+|\al|^{\frac{-2p}{2+p}}[\m B_p A_{k,p,\al}]^{\frac{-2}{p+2}}\,)^{\frac{p+2}{2}}|\al|^{\frac{p}{2}}\m B_pA_{k,p,\al}
\\&=& \m A_p|\al|^{\frac{p}{2}} |\al|^{-p}[\m B_pA_{k,p,\al}]^{-1} (1+|\al|^{\frac{2p}{2+p}}[\m B_p A_{k,p,\al}]^{\frac{2}{p+2}}\,)^{\frac{p+2}{2}}|\al|^{\frac{p}{2}}\m B_pA_{k,p,\al}
\\&=&  \m A_p (1+|\al|^{\frac{2p}{2+p}}[\m B_p A_{k,p,\al}]^{\frac{2}{p+2}}\,)^{\frac{p+2}{2}}
\\&=& \m A_p \big[1+|\al|^{\frac{2p}{2+p}}2^\frac{|p-2|}{(p+2)}{h'(r_0)}^{\frac{2(1-p)}{2+p}}(\frac{p}{k-\ga})^\frac{2p}{p+2}\,\big]^{\frac{p+2}{2}}.
\end{eqnarray*}
The last equality holds since $\m A_p \m B_p = \max\{1,2^{\frac{p-2}{2}}\} \max\{1,2^{\frac{2-p}{2}}\}=2^{\frac{|p-2|}{2}}$. We also have
\begin{eqnarray*}
\Phi_{k,p,\al} &:=& E_{k,p,\al,\ep_0}=\m A_p(|\al|+\ep_0^2)^{\frac{p}{2}}|\al|^{\frac{p}{2}} \m B_p B_{k,p,\al}
\\&=& \m A_p(|\al|+[|\al|^{\frac{p-2}{2}}\m B_p A_{k,p,\al}]^{\frac{-2}{p+2}})^{\frac{p}{2}}|\al|^{\frac{p}{2}} \m B_p B_{k,p,\al}
\\&=& \m A_p(|\al|^2+|\al|^{\frac{4}{2+p}}[\m B_p A_{k,p,\al}]^{\frac{-2}{p+2}})^{\frac{p}{2}}\m B_p B_{k,p,\al}
\\&=& \m A_p (|\al|^{\frac{2p}{2+p}}+[\m B_p A_{k,p,\al}]^{\frac{-2}{p+2}})^{\frac{p}{2}}|\al|^{\frac{2p}{2+p}}\m B_p B_{k,p,\al}
\\&=& 2^{\frac{|p-2|}{2}}\frac{\ga p}{k-\ga} (|\al|^{\frac{2p}{2+p}}+2^{\frac{-|p-2|}{p+2}}h'(r_0)^{\frac{2(p-1)}{p+2}}(\frac{p}{k-\ga})^{\frac{-2p}{p+2}})^{\frac{p}{2}}|\al|^{\frac{2p}{2+p}}
\end{eqnarray*}
and
\begin{eqnarray*}
\De_{k,p,\al,\ep} &:=& F_{k,p,\al,\ep} = \m A_p(|\al|+\ep_0^2)^{\frac{p}{2}}|\al|^{\frac{p}{2}}
\\&=& \m A_p (|\al|^{\frac{2p}{2+p}}+[\m B_p A_{k,p,\al}]^{\frac{-2}{p+2}})^{\frac{p}{2}}|\al|^{\frac{2p}{2+p}}
\\&=& \max\{1,2^{\frac{p-2}{2}}\}(|\al|^{\frac{2p}{2+p}}+2^{\frac{-|p-2|}{p+2}}h'(r_0)^{\frac{2(p-1)}{p+2}}(\frac{p}{k-\ga})^{\frac{-2p}{p+2}})^{\frac{p}{2}}|\al|^{\frac{2p}{2+p}}.
\end{eqnarray*}
Thus, it follows that 
\begin{eqnarray*}
\frac{1}{S}[\int_M\frac{\psi^{p^*}}{{h(r)}^{p^*\al}}]^{p/p^*} &\le& \Ga_{k,p,\al} \int_M [\frac{|\na^M \psi|^p}{{h(r)}^{p\al}}+\frac{\psi^p|H|^p}{p^p{h(r)}^{p\al}}]\\
&& \hspace{-1cm} - \ \Phi_{k,p,\al} \int_M \frac{\psi^p{h'(r)}|\bar\na r^\perp|^2}{{h(r)}^{p(\al+1)}} - \De_{k,p,\al}\int_M \frac{\psi^p{h'(r)}|\bar\na r^\perp|^p}{{h(r)}^{p(\al+1)}}
\nonumber,  
\end{eqnarray*}

Now, assume $\xi_0\in M$. As we have observed in the proof of Theorem \ref{teo-hardy-norm}, it holds $\vol_M ([r<2\de]\cap M)=O(\de^k)$ and $\vol_{\p M} ({\p M\cap [r<2\de]})=O(\de^{k-1})$, as $\de>0$ goes to $0$.

Consider the cut-off function $\eta=\eta_\de\in C^\infty(M)$ satisfying:
\begin{eqnarray}\label{cut-off}
&& 0\le \eta \le 1, \mbox{ in } M; \nonumber\\
&&\eta = 0, \mbox{ in } [\,r<\de]\cap M, \ \mbox{ and } \eta = 1, \mbox{ in } [\,r>2\de]\cap M;\\
&& |\na^M \eta|\le L/\de, \nonumber
\end{eqnarray}
for some $L>1$ that does not depend on $\de$ and $\eta$.

Let $\phi=\eta\psi$. Since $\phi\in C^1(M)$ and   $\xi_0\notin M'=\spp(\phi)$, it holds
\begin{eqnarray*}
\frac{1}{S}[\int_M\frac{\phi^{p^*}}{{h(r)}^{p^*\al}}]^{p/p^*} + \Phi_{k,p,\al}\int_M \frac{\phi^p{h'(r)}|\bar\na r^\perp|^2}{{h(r)}^{p(\al+1)}} + \De_{k,p,\al}\int_M \frac{\phi^p{h'(r)}|\bar\na r^\perp|^p}{{h(r)}^{p(\al+1)}}
\\&& \hspace{-6.0cm} \le \Ga_{k,p,\al} \int_M[\frac{|\na^M\phi|^p}{{h(r)}^{p\al}}+ \frac{\phi^p|H|^p}{p^p{h(r)}^{p\al}}].
\end{eqnarray*}
Notice that, 
\begin{eqnarray}\label{est-grad-0-phi-sob-1}
\int_M \frac{|\na^M\phi|^p}{{h(r)}^{\al p}} &=& \int_M \frac{|\eta \na^M\psi + \psi \na^M \eta|^p}{{h(r)}^{\al p}}
\nonumber\\&=& \int_{[r>2\de]\cap M} \frac{|\na^M\psi|^p}{{h(r)}^{\al p}} + \int_{[\de<r<2\de]\cap M} \frac{|\eta \na^M\psi + \psi \na^M \eta|^p}{{h(r)}^{\al p}},
\end{eqnarray}
and, since $h(\de)=O(\de)$, as $\de\to 0$, 
\begin{eqnarray}\label{est-grad-0-phi-sob-2}
\int_{[\de<r<2\de]\cap M} \frac{|\eta \na^M\psi + \psi \na^M \eta|^p}{{h(r)}^{\al p}} &\le& \int_{[\de<r<2\de]} (|\na^M\psi|^p+ \psi^p O(\frac{1}{\de^p}))O(\frac{1}{\de^{\al p}})
\nonumber\\&=& \int_{\{\de<r<2\de\}}(O(\frac{1}{\de^{\al p}}) + O(\frac{\de^{-p}}{\de^{\al p}}))
\nonumber\\&=& (O(\frac{1}{\de^{\al p}}) + O(\frac{\de^{-p}}{\de^{\al p}}))O(\de^{k}) 
\nonumber\\&=& O(\de^{k-\al p}) + O(\de^{k-p(\al+1)}) = O(\de^{k-\ga}),
\end{eqnarray}
as $\de\to 0$, since $k-\ga>0$. Hence, 
\begin{eqnarray*}
\frac{1}{S}[\int_{[r>2\de]\cap M} \frac{\psi^{p^*}}{{h(r)}^{p^*\al}}]^{p/p^*} 
&&\\&& \hspace{-3cm} + \ \Phi_{k,p,\al}\int_{[r>2\de]\cap M} \frac{\psi^p{h'(r)}|\bar\na r^\perp|^2}{{h(r)}^{p(\al+1)}} + \De_{k,p,\al}\int_{[r>2\de]\cap M} \frac{\psi^p {h'(r)}|\bar\na r^\perp|^p}{{h(r)}^{p(\al+1)}} \\ && \hspace{-3cm}\le \ \Ga_{k,p,\al} \int_{M}[\frac{|\na^M\psi|^p}{{h(r)}^{p\al}}+ \frac{\psi^p|H|^p}{p^p{h(r)}^{p\al}}] + O(\de^{k-\ga}).
\end{eqnarray*}
Since $k-\ga>0$, taking $\de\to 0$, Theorem \ref{teo-sob} follows.
\end{proof}

As a corollary, we have the weighted Hoffman-Spruck type inequality for submanifolds in Cartan-Hadamard manifolds.
\begin{corollary} Assume $\bar M$ is a Cartan-Hadamard manifold. We fix any $\xi_0\in \bar M$ and let $r=r_{\xi_0}=d_{\bar M}(\cdot\,,\xi_0)$. Let $1\le p<k$ and $-\infty<\al<\frac{k-p}{p}$. Then, for all $\psi\in C^1(M)$, with $\psi=0$ on $\p M$, it holds
\begin{eqnarray*}
\frac{1}{S}[\int_M\frac{|\psi|^{p^*}}{{r}^{p^*\al}}]^{p/p^*} + \Phi_{k,p,\al}\int_M \frac{|\psi|^p|\bar\na r^\perp|^2}{{r}^{p(\al+1)}} +  \De_{k,p,\al} \int_M \frac{|\psi|^p|\bar\na r^\perp|^p}{{r}^{p(\al+1)}} 
\\&& \hspace{-4.5cm} \le \Ga_{k,p,\al} \int_M[\frac{|\na^M\psi|^p}{{r}^{p\al}}+ \frac{|\psi|^p|H|^p}{p^p{r}^{p\al}}],
\end{eqnarray*}
Here, $p^*=\frac{kp}{k-p}$, $S=S_{k,p}>0$ depends only on $k$ and $p$, and $\Ga_{k,p\al}$, $\Phi_{k,p,\al}$ and $\De_{k,p,\al}$ are defined as in Theorem \ref{teo-sob}, with $h'(r_0)=1$.
\end{corollary}

\section{The Caffarelli-Kohn-Nirenberg inequality for submanifolds}

Inspired by an argument in Bazan and Neves \cite{BN}, we will obtain the  Caffarelli-Kohn-Nirenberg type inequality for submanifolds (see Theorem \ref{CKN-ineq-teo} below) by interpolating Theorem \ref{teo-hardy} and Theorem \ref{teo-sob}. In order to do that, first, we will test the interpolation argument  to prove a particular case of our  Caffarelli-Kohn-Inequality type inequality (compare with Theorem \ref{teo-sob} above). We prove the following. 
\begin{theorem}\label{sob-weighted teo} Fixed  $\xi_0\in \bar M$, assume $(\bar K_{\rad})_{\xi_0}\le \m K(r)$, where $r=d_{\bar M}(\cdot\,,\xi_0)$. 
Assume $M$ is contained in $\m B=\m B_{r_0}(\xi_0)$, with $r_0=\min\{\bar r_0(\m K), Inj_{\bar M}(\xi_0)\}$. Let $1\le p<k$ and $-\infty<\al<\frac{k-p}{p}$. Let $s>0$ and $\al\le \ga\le \al+1$ satisfying the balance condition: $$\frac{1}{s}=\frac{1}{p}-\frac{(\al+1)-\ga}{k}=\frac{1}{p^*}+\frac{\ga-\al}{k}.$$ We write $s=(1-c)p + c p^*$, for some $c\in [0,1]$. Then, for all $\psi\in C^1(M)$, with $\psi=0$ on $\p M$, it holds
\begin{equation*}
[\int_M \frac{|\psi|^s}{{h(r)}^{s\ga}}]^{\frac{p}{s}}\le (\frac{\La}{h'(r_0)})^{\frac{p(1-c)}{s}}(S\,\Ga)^{\frac{p^*c}{s}}\int_M [\frac{|\na^M\psi|^p}{{h(r)}^{p\al}}+\frac{|\psi|^p|H|^p}{p^p{h(r)}^{p\al}}],
\end{equation*}
provided either $k<7$ or $\vol(M)<\m D$, being $0<\m D\le +\infty$ a constant depending only on $r_0$ and $Inj_{\bar M}(M)$. Here, $S>0$ is a constant depending only on $k$ and $p$, and  \begin{eqnarray*}
\La&=& \max\{1,2^{\frac{p-2}{2}}\} \frac{p^p h'(r_0)^{-p}}{[k-p(\al+1)]^p},\\
\Ga&=&  \max\{1,2^{\frac{p-2}{2}}\} h'(r_0)^{1-p}\big[h'(r_0)^\frac{2(p-1)}{p+2}+|\al|^{\frac{2p}{2+p}}2^\frac{|p-2|}{(p+2)}(\frac{p}{k-p(\al+1)})^\frac{2p}{p+2}\,\big]^{\frac{p+2}{2}}.
\end{eqnarray*}
\end{theorem}

\begin{proof} We write $\ga=(1-\te)(\al+1)+\te \al$, for some $\te\in [0,1]$. Since $s=(1-c)p+cp^*$, with $c\in [0,1]$, by the balance condition, it holds $$c=\frac{\te p}{\te p+ (1-\te)p^*}=\frac{\te(k-p)}{k-\te p},$$ which implies $\te(1-c)p=(1-\te)c p^*$. Hence,  after a straightforward computation, one has $s\ga = p(1-c)(\al+1) + p^* c\al$. By the H\"{o}lder inequality,
\begin{eqnarray}
\int_M\frac{\psi^s}{{h(r)}^{s\ga}}&=&\int_M\frac{\psi^{p(1-c) + p^*c}}{{h(r)}^{p(1-c)(\al+1) + p^*c\al}} = \int_M\frac{\psi^{p(1-c)}}{{h(r)}^{p(1-c)(\al+1)}}\frac{\psi^{ p^*c}}{{h(r)}^{p^*{c\al}}} 
\nonumber\\&\leq& [\int_M(\frac{\psi^{p(1-c)}}{{h(r)}^{p(1-c)(\al+1)}})^{\frac{1}{1-c}}]^{1-c}[\int_\Si(\frac{\psi^{ p^*c}}{{h(r)}^{p^*c\al}})^\frac{1}{c}]^c 
\nonumber\\&\le& [\frac{1}{{h'(r_0)}} \int_M\frac{\psi^{p}h'(r)}{{h(r)}^{p(\al+1)}}]^{1-c}[\int_M\frac{|\psi|^{ p^*}}{{h(r)}^{p^* \al}}]^c.\nonumber
\end{eqnarray}
The last inequality holds since $1=h'(0)\ge h'(r)\ge h'(r_0)$. 
Thus, using Theorem \ref{teo-hardy} and Theorem \ref{teo-sob}, 
\begin{eqnarray*} [\int_M\frac{\psi^s}{{h(r)}^{s\ga}}]^{\frac{p}{s}} &\leq& [\La\int_M (\frac{|\na^M \psi|^p}{{h(r)}^{p\al}}+\frac{\psi^p|H|^p}{p^p{h(r)}^{p\al}})]^{\frac{p(1-c)}{s}} \times
 \nonumber\\&& \hspace{2cm}\times \, [S\,\Ga\int_M (\frac{|\na^M \psi|^p}{{h(r)}^{p\al}}+\frac{\psi^p|H|^p}{{h(r)}^{p\al}})]^{\frac{p^*c}{s}} 
\nonumber \\ &\leq& \La^{\frac{p(1-c)}{s}}(S\,\Ga)^{\frac{p^*c}{s}}\int_M(\frac{|\na^M \psi|^p}{{h(r)}^{p\al}}+\frac{\psi^p|H|^p}{{h(r)}^{p\al}}),\nonumber
\end{eqnarray*}
where $\La=\La_{k,p,\al}=\max\{1,2^{\frac{p-2}{2}}\} \frac{p^p h'(r_0)^{-p}}{[k-p(\al+1)]^p}$ and $\Ga=\Ga_{k,p,\al}$ is given as in Theorem \ref{teo-sob}. 
\end{proof}

Now, we will state our Caffarelli-Kohn-Nirenberg type inequality for submanifolds.

\begin{theorem}\label{CKN-ineq-teo} Fixed  $\xi_0\in \bar M$, assume $(\bar K_{\rad})_{\xi_0}\le \m K(r)$, where $r=d_{\bar M}(\cdot\,,\xi_0)$. 
Assume $M$ is contained in $\m B=\m B_{r_0}(\xi_0)$, with $r_0=\min\{\bar r_0(\m K), Inj_{\bar M}(\xi_0)\}$.  Let $1\le p<k$ and $-\infty<\al<\frac{k-p}{p}$. Furthermore, let $q>0,t>0$ and  $\be,\ga,\si$  satisfying
\begin{enumerate}[(i)]
\item\label{convex} $\ga$ is a convex combination, $\ga=a\si + (1-a)\be$, for some $a\in [0,1]$ and $\al\le \si\le \al+1$;
\item\label{balance1} Balance condition: $\frac{1}{t}-\frac{\ga}{k}=a(\frac{1}{p}-\frac{\al+1}{k})+(1-a)(\frac{1}{q}-\frac{\be}{k})$.
%\item\label{balance2} if $0<a<1$ then $\al\le \si \le \al+1$. 
\end{enumerate}
Then, for all $\psi\in C^1(M)$, with $\psi=0$ on $\p M$, it holds
\begin{equation}\label{CKN-inequality}
[\int_M \frac{|\psi|^t}{{h(r)}^{\ga t}}]^\frac{1}{t} \le  C\big[\int_M (\frac{|\na^M \psi|^p}{{h(r)}^{\al p}} + \frac{|\psi|^p|H|^p}{{h(r)}^{\al p}})\big]^{\frac{a}{p}}[\int_M \frac{|\psi|^q}{{h(r)}^{\be q}}]^{\frac{1-a}{q}},
\end{equation}
provided either $k<7$ or $\vol(M)<\m D$, being $0<\m D\le +\infty$ a constant depending only on $r_0$ and $Inj_{\bar M}(M)$. Here, 
$$C=(\frac{\La}{h'(r_0)})^{\frac{p(1-c)}{s}}(S\,\Ga)^{\frac{p^*c}{s}},$$ where $c\in [0,1]$ and $s\in [p,p^*]$ depend only on the parameters $p,k,\al$ and $\si$, $S$ depends only on $k$ and $p$, and $\La$ and $\Ga$ are defined as in Theorem \ref{sob-weighted teo}. 
\end{theorem}
\begin{proof} %The fact that \ref{balance1} and \ref{balance2} are necessaries conditions to obtain  (\ref{CKN-inequality}) is well known and follows by the classical proof of the CKN inequality, just considering functions $\psi\in C^\infty_0(\real^n)$. Now, we will assume $\ref{balance1}$ and \ref{balance2} hold.

If $a=1$ then $\al\le \ga=\si \le \al+1$ and $\frac{1}{t}=\frac{1}{p}-\frac{(\al+1)-\ga}{k}=\frac{1}{p^*}+\frac{\ga-\al}{k}$, in particular, $p\le t\le p^*$. Thus, Theorem \ref{CKN-ineq-teo} follows  from Theorem \ref{sob-weighted teo}. If $a=0$ then $\ga=\be$ and $q=t$, hence there is nothing to do. From now on, we will assume $0<a<1$. 

By \ref{convex} and \ref{balance1}, we obtain
\begin{eqnarray}\label{eq-1/r}
\frac{1}{t} &=& \frac{\ga}{k} + a(\frac{1}{p}-\frac{\al+1}{k})+(1-a)(\frac{1}{q}-\frac{\be}{k})
\nonumber\\&=& \frac{a\si + (1-a)\be}{k} + a(\frac{1}{p}-\frac{\al+1}{k})+(1-a)(\frac{1}{q}-\frac{\be}{k})
\nonumber\\&=& a(\frac{1}{p}-\frac{(\al+1)-\si}{k}) + \frac{1-a}{q}
\nonumber\\&=& \frac{a}{s} + \frac{1-a}{q},
\end{eqnarray} 
where $\frac{1}{s}=\frac{1}{p}-\frac{(\al+1)-\si}{k}=\frac{1}{p^*} + \frac{\si-\al}{k}$. Hence, $s=\frac{kp}{k-p[(\al+1)-\si]}\in [p,p^*]$. 

We write 
\begin{equation}\label{exp-r}
t=(1-b)q+bs. 
\end{equation}
If $s=q$, we take $b=a$. If $s\neq q$ then, by (\ref{eq-1/r}),  $t=q+b(s-q) = (\frac{a}{s} + \frac{1-a}{q})^{-1}=\frac{sq}{aq+(1-a)s}$. Hence,
\begin{eqnarray}\label{exp-b}
b&=& \frac{r-q}{s-q} = (\frac{sq}{aq+(1-a)s}-q)\frac{1}{s-q} = (\frac{sq -q(aq+(1-a)s)}{aq+(1-a)s})\frac{1}{s-q}
\nonumber\\&=& (\frac{sq - aq^2 - qs + asq}{aq+(1-a)s})\frac{1}{s-q} = (\frac{ s-q}{aq+(1-a)s})\frac{aq}{s-q} 
\nonumber\\&=& \frac{aq}{aq+(1-a)s}.
\end{eqnarray}
Thus, $b=\frac{aq}{aq+(1-a)s}\in [0,1]$, independently whether $s=q$ or not. In particular,  $(1-b)=\frac{(1-a)s}{aq+(1-a)s}$, hence $(1-b)aq + (1-b)(1-a)s = (1-a)s$, 
which implies,
\begin{equation}\label{1-b-eq}
(1-b)aq=(1-a)bs.
\end{equation}
Thus, it holds 
\begin{eqnarray}\label{exp-ga-b}
\ga t &=& (a\si+(1-a)\be)((1-b)q+bs) 
\nonumber\\&=& [(1-b)aq]\si + abs\si + (1-a)(1-b)q\be + [(1-a)b s] \be
\nonumber\\&=& [(1-a)bs]\si + abs\si + (1-a)(1-b)q\be + [(1-b)aq]\be
\nonumber\\&=& bs\si + (1-b)q\be. 
\end{eqnarray}
By (\ref{exp-r}) and (\ref{exp-ga-b}),
\begin{eqnarray}\label{CKN-est-sg}
[\int_M \frac{|\psi|^t}{{h(r)}^{\ga t}}]^\frac{1}{t} &=& [\int_M \frac{|\psi|^{(1-b)q+bs}}{{h(r)}^{(1-b)q\be+b s\si}}]^\frac{1}{t} = [\int_M \frac{|\psi|^{bs}}{{h(r)}^{b s\si}}\frac{|\psi|^{(1-b)q}}{{h(r)}^{(1-b)q\be}}]^\frac{1}{t}
\nonumber\\&\le& [\int_M (\frac{|\psi|^{bs}}{{h(r)}^{b s\si}})^\frac{1}{b}]^\frac{b}{t}[\int_M(\frac{|\psi|^{(1-b)q}}{|X|^{(1-b)q\be}})^\frac{1}{1-b}]^\frac{1-b}{t}
\nonumber\\&& \hspace{-1.2cm}= \ [\int_M \frac{|\psi|^{s}}{{h(r)}^{s\si}}]^\frac{b}{t}[\int_M\frac{|\psi|^{q}}{{h(r)}^{q\be}}]^\frac{1-b}{t}
= [\int_M \frac{|\psi|^{s}}{{h(r)}^{s\si}}]^\frac{a}{s}[\int_M\frac{|\psi|^{q}}{{h(r)}^{q\be}}]^\frac{1-a}{q}.
\end{eqnarray}
The last equality holds since, by (\ref{eq-1/r}) and (\ref{exp-b}), we obtain $\frac{b}{r}=(\frac{aq}{aq+(1-a)s})(\frac{a}{s} + \frac{1-a}{q})=\frac{a}{s}$ and $\frac{1-b}{r}=\frac{1-a}{q}$. 

Now, since $p\le s\le p^*$ satisfies,  $\frac{1}{s}=\frac{1}{p}-\frac{(\al+1)-\si}{k}$, the balance condition holds:
\begin{equation*}
\frac{1}{s}+\frac{\si}{k}=\frac{1}{p}-\frac{\al+1}{k}.
\end{equation*}
Write $s=(1-c)p+cp^*$, with $c\in [0,1]$. By Theorem \ref{sob-weighted teo}, 
\begin{equation*}
 [\int_M \frac{|\psi|^{s}}{{h(r)}^{s\si}}]^\frac{1}{s} \le C \big[\int_M (\frac{|\na^M \psi|^p}{{h(r)}^{\al p}} + \frac{|\psi|^p|H|^p}{{h(r)}^{\al p}})\big]^{\frac{1}{p}}, 
\end{equation*}
where $C$ is given as in Theorem \ref{sob-weighted teo}. Theorem \ref{CKN-ineq-teo} is proved.
\end{proof}

As a corollary, we have the Caffarelli-Kohn-Nirenberg type inequality for submanifolds in Cartan-Hadamard manifolds.
\begin{corollary}\label{CKN-CH-cor} Assume $\bar M$ is a Cartan-Hadamard manifold. We fix any $\xi_0\in \bar M$ and let $r=d_{\bar M}(\cdot\,,\xi_0)$. Let $1\le p<k$ and $-\infty<\al<\frac{k-p}{p}$. Furthermore, let $q>0,t>0$ and  $\be,\ga,\si$  satisfying
\begin{enumerate}[(i)]
\item\label{convex} $\ga$ is a convex combination, $\ga=a\si + (1-a)\be$, for some $a\in [0,1]$ and $\al\le \si\le \al+1$;
\item\label{balance1} Balance condition: $\frac{1}{t}-\frac{\ga}{k}=a(\frac{1}{p}-\frac{\al+1}{k})+(1-a)(\frac{1}{q}-\frac{\be}{k})$.
\end{enumerate}
Then, for all $\psi\in C^1(M)$, with $\psi=0$ on $\p M$, it holds
\begin{equation}\label{CKN-inequality}
[\int_M \frac{|\psi|^t}{{r}^{\ga t}}]^\frac{1}{t} \le  C\big[\int_M (\frac{|\na^M \psi|^p}{{r}^{\al p}} + \frac{|\psi|^p|H|^p}{{r}^{\al p}})\big]^{\frac{a}{p}}[\int_M \frac{|\psi|^q}{{r}^{\be q}}]^{\frac{1-a}{q}}.
\end{equation}
Here, 
$$C=\La^{\frac{p(1-c)}{s}}(S\,\Ga)^{\frac{p^*c}{s}},$$ where $c\in [0,1]$ and $s\in [p,p^*]$ depend only on the parameters $p,k,\al$ and $\si$, $S$ depends only on $k$ and $p$, and 
\begin{eqnarray*}
\La&=& \max\{1,2^{\frac{p-2}{2}}\} \frac{p^p}{[k-p(\al+1)]^p},\\
\Ga&=&  \max\{1,2^{\frac{p-2}{2}}\}\big[1+|\al|^{\frac{2p}{2+p}}2^\frac{|p-2|}{(p+2)}(\frac{p}{k-p(\al+1)})^\frac{2p}{p+2}\,\big]^{\frac{p+2}{2}}.
\end{eqnarray*}
\end{corollary}

\begin{example} There are some inequalities that derive  from Theorem \ref{CKN-ineq-teo}. For the sake of simplicity, we assume $\bar M$ is a Cartan-Hadamard manifold. Fix any $\xi_0\in \bar M$ and let $r=d_{\bar M}(\cdot\,,\xi_0)$. By Corollary \ref{CKN-CH-cor}, there exists a constant $C$, depending only on the parameters $k,p,q,t,\ga,\al$ and $\be$, such that,  
%We assume $(\bar K_\rad)_{\xi_0}\le \m K(r_\xi)$ in $\m B=\m B_{r_0}(\xi_0)$. Assume further $M$ is contained in $\m B$, $r=r_{\xi_0} \in C^2(M)$ and $\vol(M)<\m D$, where $\m D$ is a constant (as defined in Theorem \ref{CKN-ineq-teo}) that depends only on $Inj_{\bar M}(M)$ and $r_0$. Then, there exists a constant $C$ depending only on the parameters such that
for all $\psi\in C^1(M)$ with $\psi=0$ on $\p M$, the following inequality holds.
\\
\\
1 - The weighted Michael-Simon-Sobolev inequality (compare with Theorem \ref{teo-sob} and Theorem \ref{sob-weighted teo})  is obtained from Theorem \ref{CKN-ineq-teo} by taking $a=1$ (hence $\ga=\si$). In particular, if $a=1$ and $\al=0$ then, for all $\ga\in [0,1]$ and $t>0$ satisfying $\frac{1}{t}-\frac{\ga}{k}=\frac{1}{p*}$, it holds
\begin{equation*}
[\int_M \frac{|\psi|^t}{r^{\ga t}}]^\frac{p}{t} \le C \int_{M} (|\na^M\psi|^p + |\psi|^p| |H|^p),
\end{equation*}
%for all $\psi\in C^1(M)$, with $\psi=0$ on $\p M$. Here, $C$ depending only on the parameters $\ga, s, p$ and $k$, not on $M$ and $\psi$.
%2 - An weighted Michael-Simon Isoperimetric inequality. We consider $p=1$ and $\be=0$ in Theorem \ref{CKN-ineq-teo}. Let $-\infty<\al<k-1$, $q>0$, $r>0$, and $\ga$ satisfying the balance condition: \begin{equation*} \frac{1}{r} - \frac{\ga}{k}=a(1-\frac{1}{k})+\frac{1-a}{q}. \end{equation*} Then, it holds \begin{equation*} [\int_M \frac{1}{|X|^{\ga r}}]^\frac{1}{r} \le C \big[\vol(\p M)+ \int_M|\vec{H}|\big]^a \vol(M)^{\frac{1-a}{q}}. \end{equation*}  In particular, considering $M=\bar B_r$, where $\bar B_r$ denotes a closed geodesic ball of radius $r$ and centered at some point $x_0$ in $M$, since $|X|\le d_M(\cdot\,,x_0)$, we have \begin{equation*} r^\ga \vol_M(B_r)^{\frac{a(k-1)}{k}} = r^\ga \vol_M(B_r)^{\frac{1}{r} - \frac{1-a}{q}} \le C \big[\vol(\p B_r)+ \int_{B_r}|\vec{H}|\big]^a  \end{equation*}
\\
2 - Hardy type inequality for submanifolds (compare with  Theorem \ref{teo-hardy}). We take $a=1$  and $\ga=\al+1$. Hence, $\ga=\si$ and, by the balance condition, $t=p$. Thus, it holds 
\begin{equation*}
\int_M \frac{|\psi|^p}{{r}^{(\al+1)p}} \le C \int_{M} (\frac{|\na^M\psi|^p}{{r}^{p\al}} + \frac{|\psi|^p|\vec{H}|^p}{{r}^{p\al}}).
\end{equation*}
\\
3 - Galiardo-Nirenberg type inequality for submanifolds. We take $\al=\be=\si=0$. We obtain, $\ga=0$ and, for all $t>0$, satisfying $\frac{1}{t}=\frac{a}{p^*}+\frac{1-a}{q}$, with $a\in [0,1]$, it holds
\begin{equation*}
[\int_M |\psi|^t]^\frac{1}{t} \le  C\big[\int_M (|\na^M \psi|^p + |\psi|^p|H|^p)\big]^{\frac{a}{p}}[\int_M |\psi|^q]^{\frac{1-a}{q}}.
\end{equation*}
In particular, if we take $k\ge 3$, $p=2$, $q=1$, and $a={2}/(2+\frac{4}{k})$, then $\frac{1}{t}=\frac{2}{2+\frac{4}{k}} \frac{k-2}{2k} + \frac{4}{k}\frac{1}{2+\frac{4}{k}} = \frac{k-2}{2(k+2)}+\frac{2}{k+2} = \frac{1}{2}$, and the following  Nash type inequality for submanifolds holds 
\begin{equation*}
[\int_M |\psi|^2]^\frac{1}{2} \le  C\big[\int_M (|\na^M \psi|^2 + |\psi|^2|H|^2)\big]^{\frac{k}{2k+4}}[\int_M |\psi|]^{\frac{2}{k+2}}.
\end{equation*}
4 - Heisenberg-Pauli-Weyl type inequality for submanifolds. We consider $k\ge 3$ and take $t=2$, $p=q=2$, $\ga=\al=0$, $\be=-1$ and $a=\frac{1}{2}$. The parameter conditions in Theorem \ref{CKN-ineq-teo} are satisfied. So, we obtain 
\begin{equation*}
[\int_M |\psi|^2]^{\frac{1}{2}} \le C [\int_M (|\na^M \psi|^2 + |\psi|^2|H|^2)]^{\frac{1}{4}}[\int_M r^2|\psi|^2]^{\frac{1}{4}}.
\end{equation*}  
\end{example}

\section*{Aknowledgement}
The second author thanks his friends Wladimir Neves and Aldo Bazan for your suggestions and comments. %He also thanks CNPq for its financial support. 

\end{document}